\documentclass[a4, 11pt]{amsart}
\usepackage{amssymb,latexsym}
\usepackage{color}
\usepackage[left=35mm, right=35mm, top=30mm, bottom=30mm, includefoot, includehead]{geometry}
\theoremstyle{plain}
\newtheorem{theorem}{Theorem}[section]

\newtheorem{lemma}[theorem]{Lemma}
\theoremstyle{definition}
\newtheorem{definition}[theorem]{Definition}
\newtheorem{example}[theorem]{Example}

\newtheorem{remark}[theorem]{Remark}

\usepackage{enumitem}

\newcommand{\enm}[1]{\ensuremath{#1}}          %

\newcommand{\cal}[1]{\mathcal{#1}}

\newcommand{\ZZ}{\enm{\mathbb{Z}}}

\newcommand{\PP}{\enm{\mathbb{P}}}

\newcommand{\KK}{\enm{\mathbb{K}}}
\newcommand{\TT}{\enm{\mathbb{T}}}

\newcommand{\Ii}{\enm{\cal{I}}}

\newcommand{\Oo}{\enm{\cal{O}}}

\newcommand{\Ss}{\enm{\cal{S}}}

\renewcommand{\phi}{\varphi}
\renewcommand{\theta}{\vartheta}
\renewcommand{\epsilon}{\varepsilon}

\newcommand{\Union}{\bigcup}

\renewcommand{\to}[1][]{\xrightarrow{\ #1\ }}

\newcommand{\old}[1]{}

\date{}
\title[High rank loci]{Strict inclusions of high rank loci}

\author{Edoardo Ballico, Alessandra Bernardi, and Emanuele Ventura}

\address{Dipartimento di Matematica, Universit\`a di Trento, 38123 Povo (TN), Italy}
\email{edoardo.ballico@unitn.it, alessandra.bernardi@unitn.it}

\address{Dept. of Mathematics, Texas A\&M University,
College Station, TX 77843-3368, USA}
\email{eventura@math.tamu.edu, emanueleventura.sw@gmail.com}


\begin{document}

\maketitle

\begin{abstract} 
For a given projective variety $X$, the high rank loci are the closures of the sets of points whose $X$-rank is higher than the generic one. 
We show examples of strict inclusion between two consecutive high rank loci. 
Our first example is for the Veronese surface of plane quartics. Although Piene had already shown an example when $X$ is a curve, we construct infinitely many curves in $\PP^4$ for which such strict inclusion appears.  
For space curves, we give two criteria to check whether the locus of  points of maximal rank 3 is finite (possibly empty). 
\end{abstract}

\section{Introduction}

Tensor rank and symmetric tensor ranks have recently attracted a lot of attention, because of their natural appearance in several pure and applied contexts (\cite{jml, Pierre, BC, Allman, Frederic}). However, the notion of rank may be generalized to any projective variety. 

We work over the complex numbers. Let $X\subset \PP^N$ be a complex projective non-degenerate variety, and let $p\in \PP^N$ be a point. The {\it $X$-rank} of $p$ 
is defined to be: 
\[
\textnormal{rk}_X(p) = \min\left \lbrace r \ | \ p\in \langle p_1,\ldots, p_r\rangle, \mbox{ where } p_i\in X \right\rbrace.
\]
In the examples of tensor and symmetric tensor ranks the underlying varieties $X$'s are the Segre varieties and the Veronese varieties respectively.

The {\it $X$-generic rank} is the least integer $g$ such that the generic element in $\langle X \rangle$ has $X$-rank equal to $g$.

In \cite{BHMT}, Buczy\' nski, Han, Mella, and Teitler made the first systematic study of loci of points whose $X$-rank is higher than the generic one. It is worth noting that, in general, even the existence of such points is not known.
Define the locus
\[
W_k = \overline{\lbrace p\in \PP^N \ | \ \textnormal{rk}_X(p) = k \rbrace},
\] 
that is, the Zariski closure of the locus of points of $X$-rank $k$. Let $g$ be the $X$-generic rank with respect to a non degenerate variety $X\subset \PP^N$ and let $m$ be the $X$-maximal rank, i.e., $m$ is the minimum integer such that {\it every} element of $\langle X\rangle$ is in the span of at most $m$ points of $X$. 
For $k\leq g$, the set $W_k$ coincides with the secant variety $\sigma_k(X):= \overline{\cup_{p_1, \ldots , p_k\in X}\langle p_1, \ldots , p_k\rangle}$. In \cite[Theorem 3.1]{BHMT}, the authors show the following inclusion: 
\[
W_k + X \subseteq W_{k-1},
\]
where $W_k+X$ denotes the join between $W_k$ and $X$, for all $g+1\leq k\leq m$. 

\medskip

If $X = \nu_3(\PP^2)$ is the Veronese surface of plane cubics, then the maximal $X$-rank is $5$ (c.f. e.g. \cite{BGI})  and it is only attained by all 
reducible cubics whose components are a smooth conic and a tangent line to it: those whose normal form is $y(x^2+yz)$ (cf. \cite{lt}).
In this case the $X$-generic rank is $4$ (cf. \cite{AH}). Moreover, one can show that $\dim W_5 = 6$ \cite[Example 4.11]{BHMT}. In this case, we have: 
\[
W_5 + X = W_4= \sigma_4(X) = \PP^9.
\]
We will sketch an idea of the proof of this fact in Example \ref{equality}.

\medskip

One instance of strict inclusion $W_k +X \subsetneq W_{k-1}$ was already known in the literature, although stated in a different fashion. In 1981, Piene proved the existence of a smooth and non-degenerate degree $4$ rational curve $X\subset \PP^3$ with $W_3$ nonempty and finite (\cite[Case $a_1$ at p. 101]{Piene}). Since $\sigma _2(X)=\PP^3$ for any non-degenerate curve $X$, one has $W_2 =\PP^3$ and $\dim (W_3+X)=2$. We elaborate more on this in Example \ref{Piene}.

\medskip

In the present paper, we provide further examples where equality of high rank loci fails. These are featured in our main results: 

\begin{theorem}\label{main}
Let $X = \nu_4(\PP^2)\subset \PP^{14}$ be the Veronese surface of plane quartics, and let $W_7$ be the maximal $X$-rank locus. Then: 
\[
W_7 + X \subsetneq W_6=\sigma_6(X) = \PP^{14}. 
\]
\end{theorem}

\begin{theorem}\label{ii1}
Fix integers $d, g$ such that $g\ge 0$ and $d\ge 3$. Let $X\subset \PP^3$ be a smooth and non-degenerate curve of degree $d$ and arithmetic genus $g$. Assume either $23g < d^2-3d -15$, or, when $d$ is even, $g < (3d^2-16d+16)/16$.
Then $W_3$ is infinite if and only if  $g=0$ and $X$ is a smooth rational curve projectively equivalent to the curve parametrized by $x_0 = z_0^d$, $x_1= z_1z_0^{d-1}$, $x_2 = a_2z_1^{d-1}z_0$, $x_3= a_3z_1^d$ with $a_2a_3\ne 0$. Therefore, when $W_3$ is finite and nonempty,
\[
W_3 + X \subsetneq W_2 = \sigma_2(X)=\PP^3.
\] 
\end{theorem}
See Example \ref{ii0} for a discussion of the curves appearing in Theorem \ref{ii1}.

\begin{theorem}\label{curvep4}
For every $d\geq 5$, there exists a degree $d$ rational curve $X\subset \PP^4$ such that $\dim W_4\leq 1$. Thus 
\[
W_4 + X \subsetneq W_3 =\sigma_3(X) = \PP^4.
\]
\end{theorem}

\noindent {\bf Structure of the paper.} In \S 2, we briefly discuss some preliminary examples: when $X=\nu_3(\PP^2)$ and Piene's curve. In \S 3, we prove Theorem \ref{main}. In \S 4, we discuss space curves and prove Theorem \ref{ii1}. In \S 4, we construct rational curves of degree $d\geq 5$ in $\PP^4$ with $\dim W_4\leq 1$, thus proving Theorem \ref{curvep4}. \\
\vspace{5mm}

\begin{small}
\noindent {\bf Acknowledgements.} We thank J. Buczy\'{n}ski, K. Han, M. Mella, and Z. Teitler for useful discussions. 
E. Ballico and A. Bernardi were partially supported by MIUR and GNSAGA of INdAM (Italy). 
E. Ventura would like to thank the Department of Mathematics of Universit\`{a} di Trento, where part of this project was conducted, for the warm hospitality. 
E. Ventura would like to thank all the participants of the secant varieties working group held at IMPAN, especially J. Buczy\'{n}ski, for 
an inspiring atmosphere, and acknowledges IMPAN for financial support. 
\end{small}

\section{Some preliminary examples}

Let $\nu_d: \PP^n\rightarrow \PP^N$, where $N=\binom{n+d}{n}-1$, be the Veronese embedding of a projective space $\PP^n$. 
Denote a zero-dimensional degree $k$ scheme supported at a point $p\in \PP^n$  with $Z(k,p)$. For a scheme $Z$, we denote with $\langle Z\rangle$ its projective span.

\begin{definition}[\cite{BR, RS}]\label{cactus:def} Let $f$ be a homogeneous polynomial of degree $d$ in $n+1$ variables. The \emph{cactus rank} of $f$ is the minimal degree of a zero-dimensional scheme $Z\subset \mathbb{P}^n$ such that $f\in \langle \nu_d(Z)\rangle $. We say that such a $Z$ {\it evinces} the cactus rank of $f$. 
\end{definition}

\begin{example}[$X=\nu_3(\PP^2)$]\label{equality}

We sketch a possible approach to see $W_5 + X = W_4= \sigma_4(X) = \PP^9$. First, recall that every scheme $Z$ evincing the cactus rank of an element of $W_5$ is a scheme $Z(3,p)$ supported on a smooth conic. Therefore, in order to show the equality it is enough to prove that a general cubic is in the span of a $Z(3,p)$ supported on a smooth conic and a simple point $Z(1,q)$. 

Let $f\in \PP(\textnormal{H}^0(\Oo_{\PP^2}(3)))$ be a general cubic. Its symmetric rank is $4$. The variety parametrizing the schemes $Z$ of degree $4$ such that $f\in \langle Z\rangle$ is the so called {\it variety of sums of powers}, $\textnormal{VSP}(f,4)$ (see \cite{RS0} for more on these classical varieties). For a general cubic $f$, it is a classical result that $\textnormal{VSP}(f,4)\cong \PP^2$. 

For a given form $f\in  \PP(\textnormal{H}^0(\Oo_{\PP^n}(d)))$, its {\it apolar ideal} $f^{\perp}$ is the homogeneous ideal in the ring of polynomial differential operators $T = \mathbb C[\partial_0,\ldots, \partial_n]$ generated by all $g\in T$ such that $g\circ f = 0$, where $\circ$ denotes the usual differentiation. For a general cubic $f$, its apolar ideal $f^{\perp}$ is generated in degree $2$; the degree $2$ homogeneous part $\left(f^{\perp}\right)_2$ is a net of conics with empty base locus. 

The  $\textnormal{VSP}(f,4)$ can be realized as the image of the regular map $\phi_f: \PP^2\rightarrow \PP^2$ defined by the net of conics
$\left(f^{\perp}\right)_2$.  The morphism $\phi_f$ is a generically $4:1$ map and the closure of its image is $\textnormal{VSP}(f,4)$. 

The branch locus $\mathcal B_f$ of $\phi_f$ is a sextic curve, that is the dual curve of the Hessian $H(f)$ of the cubic $f$ (cf. \cite[\S 3]{mmsv}). Following the classical De Paolis algorithm (cf. \cite{DeP}), fix the tangent line $q$ to one of the nine flexes of $H(f)$. Then $\deg(q\cap H(f))=3$ and the intersection is supported at one point $p$. Under duality, the line $q$ corresponds to a cusp $c_q\in \mathcal B_f$. Its fiber $\phi_f^{-1}(c_q)$ is the scheme $Z = Z(3,p)+Z(1,q)$. By construction, $Z$ is the intersection of two conics from the vector space $\langle \left(f^{\perp}\right)_2\rangle$ and so spans $f$. 
\end{example}

\begin{example}[{\bf Piene's curve \cite{Piene}}]\label{Piene}
Consider a general projection of the rational normal curve  $\nu_4(\PP^1)\subset \PP^4$ onto a rational curve $X\subset \PP^3$. Recall that the maximal $X$-rank is $3$. By \cite[Case $a_1$, p. 101]{Piene}, the only cuspidal planar projection of $X$ is obtained by projecting $X$ from the {\it unique} point $p\in \PP^3$ of intersection of an ordinary tangent with a stall tangent. For the reader who is not familiar with this notation, a stall tangent  is the tangent line at {\it stall} point of $X$, i.e., a point whose local parametrization is: 
\begin{align}\label{eqa1}
& x = at+\cdots, \notag \\
& y = b_2t^{2} + b_3t^3+\cdots, \\
& z = c_4t^{4} + c_5t^5 + \cdots\notag,
\end{align}
with $ab_2c_4\neq 0$ and $b_2c_5\neq b_3c_4$ (see \cite[Lemma 1, p. 98]{Piene}). The projected curve is a rational planar quartic curve with an ordinary cusp (arising from the ordinary tangent) and a ramphoid cusp of the $1$st type \cite[Remark, p. 98]{Piene} (arising from the stall tangent). Since a projection from a point $p\in \PP^3$ is injective if and only if $\textnormal{rk}_X(p)>2$, we have that $W_3 = \lbrace p\rbrace$. Therefore $W_3 + X \subsetneq \sigma_2(X) = \PP^3$.\end{example}

\section{Proof of Theorem \ref{main}}

In this section, we prove Theorem \ref{main}. Here  $X= \nu_4(\mathbb{P}^2)$ is the Veronese surface of plane quartics. Recall that the {\it $k$-cactus variety of $X$} is the closure of the union of the scheme-theoretic linear spans of all zero-dimensional subschemes of $X$ of degree at most $k$ (cf. \cite{BucBuc,BR,BJMR}).

Since $W_6=\sigma_6(X) = \mathbb{P}^{14}$, in order to prove the inclusion between $W_7+X$ and $W_6$ it is sufficient to show that $W_7+X$ does not fill the ambient space. Since the dimension of any join is at most the sum of the dimensions of the varieties involved plus one, to prove that  $W_7+X \neq \mathbb{P}^{14}$, it is enough to show that  $\dim W_7 \leq 10$.

Notice that plane quartics of border rank 5 are expected to fill the ambient space; however it is a classical result that they fail to do that, whereas $W_6=\mathbb{P}^{14}$ (cf. \cite{AH}). Moreover, by \cite{BR}, also the $6$-th cactus variety of $\nu_4(\mathbb{P}^2)$ fills the ambient space. Now, Kleppe's classification \cite[Chapter 3]{Kleppe} of quartics in $W_7$ shows that there is no $f\in W_7$ of symmetric border rank 6. All the normal forms of \cite{Kleppe} with symmetric rank 7 were also classified in \cite{BGI} and they all turn out to have cactus rank smaller or equal than 5.

As in \S 2, we denote a zero-dimensional degree $k$ scheme supported at a point $p\in \PP^2$  with $Z(k,p)$. In \cite[Theorem 44]{BGI} the stratification by symmetric rank of $\sigma_s(X)\setminus \sigma_{s-1}(X)$ for $s=2,3,4,5$, is derived. Symmetric rank 7 arises in cactus ranks 3, 4 and 5.

\begin{itemize}

\item For cactus rank 3, there is one possible scheme: 
\begin{enumerate}
\item[(I)] $Z= Z(3,p)$ contained in a smooth conic. 
\end{enumerate}

\item For cactus rank 4, there are the two possible schemes: 

\begin{enumerate}

\item[(IIa)] $Z = Z(2,p)+ Z(2,\ell)$ (two $2$-jets supported at $p,\ell$); 

\item[(IIb)] $Z = Z(2,p) + Z(1,\ell) + Z(1,q)$ (a $2$-jet supported at $p$ and two simple points $\ell,q$).

\end{enumerate}

\item For cactus rank 5, there are three possible schemes: 

\begin{enumerate}

\item[(IIIa)] $Z = Z(3,p) + Z(2,\ell)$ (a $3$-jet supported at $p$ and a $2$-jet supported at $\ell$);

\item[(IIIb)] $Z= Z(4,p) + Z(1,\ell)$ contained in a double line: its homogeneous ideal in $S = \bigoplus_{k\geq 0} \textnormal{H}^{0}(\Oo_{\PP^2}(k))$ is of the form $\mathcal I_Z = (Q, y^2)\cap (x-z,y)$, where $Q$ is either a smooth conic whose tangent line at $p$ coincides with $\lbrace y=0\rbrace$, or a reducible conic with vertex $p$;

\item[(IIIc)] $Z = Z(2,p)+ Z(2,\ell) + Z(1,q)$ contained in a double line: its homogeneous ideal in $S = \bigoplus_{k\geq 0} \textnormal{H}^{0}(\Oo_{\PP^2}(k))$ is of the form $\mathcal I_Z = \left(x, y\right)\cap (x^2-z^2,y^2)$.

\end{enumerate}

\end{itemize}

Let $\textnormal{Hilb}_k(\PP^2)$ be the Hilbert scheme parametrizing zero-dimensional schemes $Z\subset \PP^2$ of degree $k$. Let  $\Upsilon_7$ be the subset of 
\[
\mathcal H = \Union_{3\leq k\leq 5}\textnormal{Hilb}_k(\PP^2) \times \PP^{14} = \left\lbrace (Z, f), \mbox{ where } Z\subset \PP^2, \deg(Z)=k, f\in \langle \nu_4(Z)\rangle \right\rbrace
\]
consisting of all $(Z, f)\in \mathcal H$, such that $\textnormal{rk}_X(f) = 7$. The set $\Upsilon_7$ comes equipped with a natural structure of algebraic variety. 

Then $W_7$ is the projection of $\Upsilon_7$ to $\mathbb{P}^{14}$ and so $\dim W_7 \leq \dim \Upsilon_7$. Thus proving $\dim \Upsilon_7 \leq 10$ will finish the proof.

For each one of the schemes above, we give an upper bound for the number of parameters involved: 

\begin{enumerate}

\item[(I)] $(Z, \langle Z\rangle)$ is parametrized by the choice of a conic $C\in \PP^5$ ({\bf +5}), a point $p\in C$ ({\bf +1}), and its span ({\bf +2}). Thus this gives a parameter space of dimension at most $8$. 

\item[(IIa)] $(Z, \langle Z\rangle)$ is parametrized by the choice of two points $p,\ell\in \PP^2$ ({\bf +4}), two lines passing through each of them ({\bf +2}), and its span ({\bf +3}). Thus this gives a parameter space of dimension at most $9$. 

\item[(IIb)] $(Z, \langle Z\rangle)$ is parametrized by the choice of three points $p,\ell, q\in \PP^2$ ({\bf +6}), a line passing through one of them ({\bf +1}), and its span ({\bf +3}). Thus this gives a parameter space of dimension at most $10$.

\item[(IIIa)] $(Z, \langle Z\rangle)$ is parametrized by the choice of two lines ({\bf +4}), one point on each of them ({\bf +2}), and its span ({\bf +4}). 
Thus this gives a parameter space of dimension at most $10$.

\item[(IIIb)] Suppose $Q$ is reducible. Since $Q$ has vertex $p$, the parameter space for $Q$ is $\PP^2$. In this case, 
$(Z, \langle Z\rangle)$ is parametrized by the choice of a line ({\bf +2}), two points $p,\ell$ on it ({\bf +2}), a reducible quadric with vertex at $p$ ({\bf +2}), and its span ({\bf +4}). Thus this gives a parameter space of dimension at most $10$.

Suppose $Q$ is a smooth conic and so $Q\in \PP^5$. Choose a smooth conic $Q$ and a point $p\in Q$. The tangent line $\lbrace y=0\rbrace$ at $p$ to $Q$ is determined by $Z(4,p)$. So far we have $6$ parameters. However, note that $h^0(\mathcal I_{Z(4,p)}(2)) = 2$ and 
hence there is a $\PP^1$ of generically smooth conics providing the same scheme $Z(4,p)$. Thus the parameters are in fact $5$ ({\bf +5}). Now choose a point $\ell \in \lbrace y=0\rbrace$ ({\bf +1}), and the span of $Z$ ({\bf +4}). Hence the parameter space for $(Z, \langle Z\rangle)$ has dimension at most $10$.

\item[(IIIc)] $(Z, \langle Z\rangle)$ is parametrized by the choice of one line ({\bf +2}), three points $p,\ell,q$ on it ({\bf +3}), and its span ({\bf +4}). Thus this gives a parameter space of dimension at most $9$.

\end{enumerate}

Consequently, all the cases show $\dim \Upsilon_7\leq 10$. \\
Then $\dim W_7\leq\dim \Upsilon_7\leq 10$ and so $\dim (W_7+ X)\leq 13 < \dim W_6 = 14$.

\section{Space curves}

For a zero-dimensional scheme $Z$, let $\sharp Z$ denote the cardinality of its support. We start by discussing the rational curves appearing in Theorem \ref{ii1}.

\begin{example}\label{ii0}
Let $X\subset \PP^3$ be a smooth rational curve of degree $d\geq 4$. Assume the existence of a line $L\subset \PP^3$ such that $\deg (X\cap L)\geq d-1$ and $(X\cap L)_{\mathrm{red}}=\{p\}$, i.e., it is set-theoretically a single point. Therefore $\textnormal{rk}_X(q)=3$ for any  $q\in L\setminus \{p\}$.

\vspace{0.1 cm}

To see this, first notice that  $\textnormal{rk}_X(q)>1$ since $q\in L\setminus X\cap L$. By \cite[Proposition 5.1]{lt} we have $\textnormal{rk}_X(q)\le 3$. Assume $\textnormal{rk}_X(q)=2$ and take $A\subset X$ such that $\sharp A =2$ and $q\in \langle A\rangle$. Since, by hypothesis,  $\sharp(X\cap L)=1$,  $L$ cannot coincide with $  \langle A\rangle$, moreover  $q\in \langle A\rangle \cap L$ therefore 
$ \langle A\rangle \cup L$ spans a plane $\Pi$. If $p\notin A$, then $\deg (X\cap\Pi)\ge 2+\deg (X\cap L)>d$, a contradiction with $\deg X=d$. If $p\in A$, then $\langle A\rangle \cap L =\{p\}$, because the lines are distinct, i.e., $L\ne  \langle A\rangle$. Thus $q=p$, a contradiction.

\vspace{0.1cm}

{\bf Construction.} All triples $(X,L,p)$  as in the example above can be constructed as follows. The aim is to give a degree $d$ embedding $\phi: \PP^1\to \PP^3$ and $v\in \PP^1$ such that the tangent line $L=T_{p}X$ has order of contact $d-1$ with $X$ at $p=\phi(v)$. 

Let $z_0,z_1$ be homogeneous coordinates of $\PP^1$, and $x_0,x_1,x_2,x_3$ be the ones of $\PP^3$. Up to projective automorphisms of $\PP^1$ and $\PP^3$ (the latter one is the automorphism acting on the Grassmannian of the targets  of the projection, $\mathbb G(3,\PP^d)$), we may assume $v =(1:0)$, $p = (1:0:0:0)$, $T_pX = \{x_2=x_3 =0\}$, and $O_p X=\{x_3=0\}$ be the tangent line and the osculating plane of $X$ at $p$ respectively.
Up to the above actions, we may further assume the morphism $\phi$ is defined by $x_0 = z_0^d$, $x_1= a_1z_1z_0^{d-1}$, $x_2 = a_2z_1^{d-1}z_0+bz_1^d$, $x_3= a_3z_1^d$ with $a_1a_2a_3\ne 0$. Since $a_3\ne 0$, we may further reduce to the case $b=0$. Taking the automorphism $(z_0:z_1)\mapsto (z_0:tz_1)$, with $t = a_1^{-1}$, we may also assume $a_1=1$.
Conversely, any parametrization defined by $x_0 = z_0^d$, $x_1= z_1z_0^{d-1}$, $x_2 = a_2z_1^{d-1}z_0$, $x_3= a_3z_1^d$ with $a_2a_3\ne 0$ gives the desired rational curve. 
\end{example}

\vspace{0.2cm}

\begin{proof}[Proof of Theorem \ref{ii1}:]
One direction is clear from Example \ref{ii0}. For the converse, assume $W_3$ is infinite. 
Since $W_3$ is a closed algebraic subset of $\PP^3$, there is an irreducible curve $\Gamma \subset \PP^3$ such that
$\Gamma \subseteq W_3$ and $\textnormal{rk}_X(q)=3$ for a general $q\in \Gamma$. Recall that $W_4 =\emptyset$ by \cite[Proposition 5.1]{lt}. Let $\tau (X) \subset \PP^3$ denote the tangential variety of $X$, i.e., (since $X$ is smooth)
the union of all tangent lines $T_pX$, $p\in X$. This is an integral surface containing $X$ in its singular locus. Take an arbitrary point $a\in \PP^3$ such that $\textnormal{rk}_X(a) >2$. Since $\sigma _2(X)=\PP^3$,
one necessarily has $a\in \tau (X)\setminus X$; hence there exists $p\in X$ such that $a\in T_pX$. 

Fix a general $q\in \Gamma\subseteq W_3$. Let $\pi _q: \PP^3\setminus \{q\} \to \PP^2$ denote the linear projection away from $q$. Since $\textnormal{rk}_X(q)>1$, $q\notin X$ and so
$\pi_{q|X}$ is a morphism. Set $D_q:= \pi_q(X)$. Note that, since $\textnormal{rk}_X(q)>2$, the map $\pi _{q|X}$ is injective. Thus $D_q$ is a degree $d$ plane curve with geometric genus $g$ (its normalization is $X$).

Call $\TT(q)$ the set of all $p\in X$ such that $q\in T_pX$. Since $\pi _{q|X}$ is injective, we have $\mathrm{Sing}(D_q) =\pi_q(\TT(q))$. Fix $p\in \TT(q)$
(we do not claim that such a $p$ is unique). Again, since $\pi _{q|X}$ is injective on points,  we have that set-theoretically $\{p\} =T_pX\cap X$.

Let $t$ be a uniformizing parameter of the complete local ring $\hat{\Oo}_{X,p}\cong \KK[[t]]$. We may choose an affine coordinate system $x, y, z$ around $p$ such that $X$  is locally given by the formal power series: 
\begin{align}\label{eqa1}
& x = at^{l_0+1}+\cdots, \notag \\
& y = bt^{l_1+2}+\cdots, \\
& z = ct^{l_2+3} +\cdots\notag,
\end{align}
where $abc \ne 0$ and $0\le l_0\le l_1\le l_2$; see, e.g., \cite[\S 2]{Piene}. (The smoothness of $X$ at $p$ is equivalent to $l_0 =0$.) In these coordinates, $T_pX = \{y=z =0\}$.  Two possibilities arise: either $\Gamma =T_pX$ or $\Gamma \ne T_pX$. 

\begin{enumerate}[label=(\Roman*)]
\item\label{(Iroman)} Assume $\Gamma =T_pX$. The linear projection $\pi_{T_pX}: \PP^3\setminus T_pX \to \PP^1$ induces a non-constant morphism $\psi: X\setminus T_pX \to \PP^1$. Since
$X$ is smooth, $\psi$ extends to a non-constant morphism $\psi': X\to \PP^1$. Since $\{p\} =T_pX\cap X$ (set-theoretically), we have  $\deg (T_pX\cap X)=l_1+2$ and hence $\deg (\psi') =d-l_1-2$. Since
$X\ne  T_pX$, we have $\deg (\psi')\ge 1$.

\begin{enumerate}
\item\label{(a1)} Assume $\deg (\psi') =1$, i.e., $\psi': X\to \PP^1$ is a birational morphism. Thus $g=0$. We obtain $\deg (T_pX\cap X) =d-1$ and hence we are in the assumptions of Example \ref{ii0}.  

\item\label{(a2)} Assume $\deg (\psi')\ge 2$. Since in characteristic zero the affine line $\mathbb A^1$ is algebraically simply connected (i.e., it does not admit any nontrivial \'{e}tale covering), the morphism $\psi'$ has at least two distinct ramification points. Thus, besides $\{p\} = X\cap T_pX$, there exists a tangent line $R_q$ to $X$, such that $R_q\ne T_pX$ and $q\in R_q$. Thus $\{q\} =T_pX\cap R_q$. Varying $q$ in $T_pX$, we derive that $T_pX$ meets a general tangent line of $X$, i.e., the differential of $\psi'$ is identically zero on $X$, a contradiction as we are in characteristic zero. 
\end{enumerate}

\item\label{(IIroman)} Assume $\Gamma \ne T_pX$. This means that a single $T_pX$ cannot contain the curve $\Gamma\subseteq W_3$. Therefore, varying $q\in \Gamma$, the sets $\TT(q)$ cover an open subset of $X$. Hence for a general $q\in \Gamma$, we find a point $p\in \TT(q)$, whose sequence $(l_0,l_1,l_2)$ as in \eqref{eqa1} is $(0,0,0)$: 

\vspace{0.2 cm}

\begin{quote}
{\it Claim.} For each $p\in \TT(q)$, the sequence $(l_0,l_1,l_2)$ is $(0,0,0)$.

\vspace{0.2cm}

{\it Proof of the Claim.} Let $\mathcal B$ be the set of all $a\in X$ such that the sequence in \eqref{eqa1} is not $(0,0,0)$, i.e., such that the osculating plane of $X$ at $a$ has order of contact $>3$ with $X$ at $a$ (such a plane is also called {\it non-ordinary osculating plane}). The set $\mathcal B$ is finite. If it is empty, there is nothing to prove. Otherwise, we obtain the existence of $a\in \mathcal B$ such that $q\in T_a X$ for a general $q\in \Gamma$. Thus $T_aX =\Gamma$ and so we are in case \ref{(Iroman)}, which leads to a contradiction. 
\end{quote}
\end{enumerate}
By the {\it Claim}, $\pi_q({p})$ is an ordinary cusp of the plane curve $D_q$. 
By step \ref{(IIroman)}, $D_q$ is an integral degree $d$ plane curve with only ordinary cusps as singularities. Since $D_q$ has arithmetic genus $(d-1)(d-2)/2$ and geometric genus $g$, it has $(d-1)(d-2)/2 - g$ (ordinary) cusps. 
If $(d-1)(d-2)/2 -g >  (21g+17)/2$ (i.e., $23g < d^2-3d -15)$, we derive a contradiction by the results in \cite{tono}. (In fact, Tono's result is stronger, because it bounds the number of cusps in term of the geometric genus $g$,
without requiring the cusps being ordinary.) 

Since our cusps are ordinary, there are other upper bounds for their number $\kappa$; see \cite{hir, ivi}. (When the plane curve is rational, and one has the parameterization, there are algorithms to describe its singularities, see e.g. \cite{BGI2}.) In particular, if $d$ is even, \cite[eq. 16]{hir} gives
\begin{equation}\label{eqa2}
\kappa \le d(5d-6)/16.
\end{equation}
Since in our case $\kappa = (d-1)(d-2)/2 -g$, for $d$ even, we obtain $16g \ge 3d^2-16d+16$, contradicting our assumption.
\end{proof}

By Castelnuovo's upper bound for non-degenerate curves \cite[3.12, 3.13, 3.14]{he}, the bounds on the arithmetic genus featured in Theorem \ref{ii1} are quite good, although not optimal. 

\begin{remark}
For $g =1$, we cover all even integers $\ge 6$, whereas we know that for $(d,g) =(4,1)$ we have $\dim W_3 >0$. For $g=1$ and arbitrary $d$, we cover all  $d\ge 9$.
\end{remark}

Set $W^0_3:= \{q\in \PP^3\mid \textnormal{rk}_X(q)=3\}$. As before, recall that for any $o\in \PP^3$, $\pi_o: \PP^3\setminus \{o\}\to \PP^2$ denotes the linear projection away from $o$, and (since $o\notin X$), $\phi _o:= \pi _{o|X}$ is a morphism $\phi _o: X\to \PP^2$. 

\begin{remark} 
As noticed in the proof of Theorem \ref{ii1}, the morphism $\phi_o$ is injective if and only if $o\in W^0_3$. Hence, if $o\in W^0_3$, 
$\phi _o(X)$ is a plane curve of degree $\deg (X)$ and geometric genus $p_a(X)$ with only unibranch singularities.
\end{remark}

Let $\Sigma (X)$ denote the set of all $o\in W^0_3$ such that $\phi _o(X)$ has only ordinary cusps as singularities.

\begin{theorem}\label{thx}
$W^0_3\setminus \Sigma (X)$ is infinite if and only if $X$ is as in Example \ref{ii0}.
\end{theorem}

\begin{proof}
Fix $o\in W^0_3\setminus \Sigma (X)$. By assumption, $\phi _o(X)$ is a plane curve with degree $\deg (X)$, only unibranch
singularities, but with at least one non-ordinary cusp. Since $X$ is smooth and $\phi _o$ is induced by a linear projection
away from $o\notin X$, this non-ordinary singularity of $\phi _o(X)$ corresponds to some $p\in X$ such that $o\in T_pX$
and the osculating plane $O_pX$ to $X$ at $p$ has order of contact $\ell_2+3>3$ with $X$ at $p$ \cite[\S 2]{Piene}, i.e., $O_pX$ is a non-ordinary osculating plane.  

Assume $\Sigma (X)$ is infinite. Since $\Sigma(X)$ is a constructible set and $X$ has only finitely many non-ordinary osculating planes, we obtain the existence of $p\in X$ such that the tangent line $L:= T_pX$ is contained in $\overline{\Sigma(X)}$ and $\Sigma (X)$ contains all but finitely many $o\in L$, i.e., there exists a finite set $A\subset L$ such that $W^0_3\supseteq L\setminus A$. By definition of $W^0_3$, we also see that $L\cap (X\setminus \{p\}) =\emptyset$. Hence $\ell_1+2 := \deg (L\cap X)$ is the order of contact of $L$ and $X$ at $p$. 

Now the proof proceeds exactly as the proof of Theorem \ref{ii1}: once we take the linear projection away from $L$, we derive a contradiction. 
Hence $\Sigma(X)$ is finite (possibly empty). Apply Theorem \ref{ii1} and conclude. 
\end{proof}

\section{Infinitely many curves in $\PP^4$}\label{Se}

This section is devoted to prove Theorem \ref{curvep4}. For all integers $d\ge 5$, we construct a smooth, rational and non-degenerate degree $d$
curve $X\subset \PP^4$ such that  $\dim W_4 =1$. More precisely, we show $\dim W_4 <2$
and that the locus $W_4$ contains a line.

Let $\mathbb F_1\subset \PP^4$ be the smooth rational ruled (cubic) surface, the first Hirzebruch surface \cite[\S V.2]{h}; this is also the projection of
the degree $4$ Veronese surface $\nu_2(\PP^2)\subset \PP^5$ from a point on itself. Its Picard group is $\mathrm{Pic}(\mathbb F_1)\cong \ZZ^2$; we may take as a basis of $\mathrm{Pic}(\mathbb F_1)$, a fiber $f$ of the ruling of $\mathbb F_1$ and the only integral curve $C_0\subset \mathbb F_1$ with normal bundle of degree $C_0^2=-1$. Note that $f^2=0$.  The curve $C_0$ is a section of the ruling of $\mathbb F_1$ and so $C_0\cdot f=0$. 
The chosen embedding of this surface is the one given by the complete linear system $|C_0+2f|$. 
In this embedding, the section $C_0$ has degree one, i.e., it is a line. Moreover, all fibers $\ell\in |f|$ are lines and $\mathbb F_1$ contains no other line. 

We will construct a curve $X\subset \mathbb F_1$ such that $W_4\supseteq C_0$ and $\dim W_4<2$.
Take any $Y\in |C_0+(d-1)f|$. We have $\deg (Y) = (C_0+(d-1)f)\cdot (C_0+2f) = d$. 
Notice that if $C_0$ is not a component of $Y$ we have $\deg(C_0\cap Y) =C_0\cdot (C_0+(d-1)f) =d-2$. Furthermore, if $\ell \in |f|$ and $\ell$ is not a component of $Y$, then $\deg (\ell \cap Y)= f\cdot (C_0+(d-1)f) =1$. 

When $Y$ is integral, it is a smooth rational curve of degree $d$. Indeed, the genus formula on the surface $\mathbb F_1$ implies $p_a(Y) = 1 + \frac{1}{2}\left(Y\cdot Y+ Y\cdot K_{\mathbb F_1}\right)=0$.

For all $d\ge 4$, $Y$ spans $\PP^4$, because no element of $|C_0+2f|$ contains $Y$. 
We have $h^0(\Oo _{\mathbb F_1}(C_0+(d-1)f)) = d+d-1$, i.e., $\dim |C_0+(d-1)f| =2d-2$. 

\begin{lemma}\label{claim1}
Fix $o\in C_0$ and call $E\subset C_0$ the degree $(d-2)$ zero-dimensional subscheme whose support is $\{o\}$. We have $\dim |\Ii _E(C_0+(d-1)f)| =d$ and a general $X\in |\Ii _E(C_0+(d-1)f)|$ is smooth.
\end{lemma}
\begin{proof}
The ideal sheaf exact sequence gives $\dim |\Ii _E(C_0+(d-1)f)|\ge \dim |C_0+(d-1)f | - \deg (E)= 2d-2-(d-2)=d$. Since
$h^1(\mathbb F_1,\Oo_{\mathbb F_1}(C_0+(d-1)f)) =0$, we have $\dim |\Ii _E(C_0+(d-1)f)| =d$ if and only if $h^1(\mathbb F_1,\Ii _E(C_0+(d-1)f)) =0$. Since
$E\subset C_0$, we have an exact sequence
\begin{equation}\label{eqx1}
0 \longrightarrow \Oo _{\mathbb F_1}((d-1)f)\longrightarrow \Ii _{E}(C_0+(d-1)f)\longrightarrow \Ii _{E,C_0}(C_0+(d-1)f)\to 0. 
\end{equation}
Using \cite[Lemma 2.4, V]{h}, we have $h^1(\Oo _{\mathbb F_1}((d-1)f))=0$ because $d\ge 2$. Since $C_0\cong \PP^1$ and $\Oo _{C_0}(C_0+(d-1)f)$ has degree $d-2$, we have $h^1(C_0,\Ii _{E,C_0}(C_0+(d-1)f))=0$, because $\deg(E)=d-2$. We conclude by using the cohomology exact sequence of (\ref{eqx1}). 

Now, let $X$ be a general element of $|\Ii_E(C_0+(d-1)f)|$. By the genus formula, to prove that $X$ is smooth, it is sufficient to prove that it is irreducible, i.e., the set of all reducible $Y\in |\Ii_E(C_0+(d-1)f)|$ has dimension $<d$. First we consider all the reducible $Y\in |\Ii_E(C_0+(d-1)f)|$ containing $C_0$. They are of the form $Y =C_0\cup B$ with $B\in |(d-1)f|$. Hence this set has dimension $d-1$. Next we consider the set of all reducible $Y\in |\Ii_E(C_0+(d-1)f)|$ without $C_0$ as component. There is an integer $b$ such that $1 \le b \le d-2$ such that $Y =B\cup Y'$
with $B \in |bf|$ and $Y'$ a smooth element of $|C_0+(d-1-b)f)|$. Since $\deg (Y'\cap C_0) = d-2-b$, $\deg (B\cap C_0) = b$
and $\deg (E)=d-2$, $B$ has to contain the support of $E$ and so  $B= bf_o$, where $f_o$ is the unique element of $|f|$ containing $o$. In other words, $B$ is uniquely determined by $o$ and hence it is sufficient to prove that the set of all $Y'$ has dimension $\le d-1$. The residual scheme of $E$
with respect to $B$ is the degree $(d-2-b)$ divisor $E'$ of $C_0$ with support $\{o\}$. Thus $Y'\in |\Ii
_{E'}(C_0+(d-1-b)f|$. However, the first paragraph of the proof shows $\dim  |\Ii_{E'}(C_0+(d-1-b)f)|= d-b\leq d-1$. 
\end{proof}

\begin{lemma}\label{claim2}
Keeping the notation from above, fix any smooth $X\in |\Ii _E(C_0+(d-1)f)|$. Then $\textnormal{rk}_X(q) =4$, for all $q\in
C_0\setminus \{o\}$ and $W_4 \supseteq C_0$.
\end{lemma}
\begin{proof}
By \cite[Proposition 5.1]{lt} we have $\textnormal{rk}_X(q)\le 4$. Assume $\textnormal{rk}_X(q)\le 3$ and take a subscheme $S\subset X$ such that $\sharp S\le 3$ and $q\in \langle S\rangle$, evincing its rank. Since $d\ge 5$
and $C_0$ is a line, one has $\langle E\rangle =C_0$. 

First, assume $o\notin S$ and $\sharp S=3$ (i.e., $\langle S\rangle$ is
a plane). Since $o$ is the only point of $X$ contained in $C_0$, $\deg (E\cup S) =d+1$. Since $\deg (X)=d$, we derive $\langle E\cup S\rangle =\PP^4$. Hence the line
$\langle E\rangle$ and the plane $\langle S\rangle$ must be disjoint, contradicting the assumption $q\in \langle S\rangle \cap C_0$.

Assume $o\notin S$ and $\sharp S =2$ (i.e., $\langle S\rangle$ is a line). Since $\deg (E\cup S)=d$, the span $\langle E\cup S\rangle$ is either a hyperplane or $\PP^4$, contradicting the assumption that the lines $\langle E\rangle$ and $\langle S\rangle$ contain $q$. 

Assume $o\in S$ and $\sharp S =3$. The span $\langle S\rangle$ is a plane containing $o$ and $q$ and hence containing $C_0$. Therefore $\deg (E\cup S) =d$ and hence $\dim \langle E\cup S\rangle \ge 3$, contradicting the inclusions $E\subset C_0\subset \langle S\rangle$. The other cases are analogous and left to the reader. The second assertion  follows from the first one, because $C_0$ is the closure of $C_0\setminus \{o\}$. 
\end{proof}

\begin{remark}
Since $\sigma _3(X)=\PP^4$ for any non-degenerate curve $X\subset
\PP^4$, Lemma \ref{claim1} and Lemma \ref{claim2} show that in order to construct the desired example it is sufficient to prove the existence of a smooth $X\in |\Ii_E(C_0+(d-1)f)|$ such that $\dim W_4 <2$. 
\end{remark}

Henceforth we take a general $X\in |\Ii _{E}(C_0+(d-1)f)|$. 
From now on, we assume $\dim W_4 \ge 2$ and take a projective irreducible surface $\Gamma \subseteq W_4$. 
Lemma \ref{x1} will provide a contradiction. 

For any $a\in X\setminus \{o\}$, let $\pi _a: \PP^4\setminus \{a\}\to \PP^3$ denote the linear projection from $a$.
Call $G\subset \PP^3$ the closure in $\PP^3$ of $\pi _a(\mathbb F_1\setminus \{a\})$. 
Since $\deg (\mathbb F_1)=3$ and $\mathbb F_1$ is non-degenerate and irreducible, $G$ is an irreducible
quadric surface. Note that $\pi _a$ contracts the line $f_a\in |f|$ containing $a$. We see that $\pi_{a|\mathbb F_1\setminus
f_a}$ is an embedding. Thus $\pi _{a|X\setminus \{a\}}$ is an embedding. Since $a$ is a smooth point of $X$, one has $\deg (X')
=d-1$. Moreover, $\pi _{a|X\setminus \{a\}}$ extends to a morphism $\phi : X\to \PP^3$ with the image $a':= \phi (a)$
corresponding to the tangent line $ T_aX$ of $X$ at $a$. 

Set $X':= \phi (X)$. Note that $T_aX\ne f_a$, because $X\cdot f_a = 1$. We have $a'\notin \pi
_a(X\setminus\{a\})$, because $\deg (T_aX\cap \mathbb F_1)=2$, since $T_aX\ne f_a$ (and hence $T_aX\nsubseteq \mathbb F_1$) and
$\mathbb F_1$ is scheme-theoretically cut out by quadrics. For a general $a\in X$ (it is in fact sufficient to assume that the
osculating hyperplane to $X$ at $a$ has order of contact $3$ with $X$ at $a$), $a'$ is a smooth point of $X'$. Thus $X'$ is a smooth rational curve of degree $d-1 \ge 4$. By \cite[Ex. V.2.9]{h}, $G$ is a smooth quadric surface. Up to a choice of rulings of $G$, we have $X'\in |\Oo _G(1,d-2)|$. 

The surface $\Gamma\subseteq \PP^4$ is not a cone for a general $a\in X$, because $X$ spans $\PP^4$, and the vertex of a cone
is a linear subspace of the ambient space. Thus, for a general $a\in X$, $\Gamma_a = \overline{\pi_a(\Gamma\setminus \{a\})}\subset \PP^3$ is a surface. Henceforth, assume $a\in X$ general. 

\begin{lemma}\label{o1}
We have $\dim W_3(X')\le 1$.
\end{lemma}

\begin{proof}
Since $X'$ is a smooth non-degenerate curve, its first secant satisfies $\sigma _2(X') =\PP^3$ and $\textnormal{rk}_X(q)=2$ for all $q\in \PP^3\setminus \tau(X')$. Since $\tau (X')$ is an irreducible surface, it is sufficient to prove that $\textnormal{rk}_{X'}(q) \le 2$ for a
general $q\in \tau (X')$. Fix a general $p\in X'$. In particular, we assume that the order of contact of $T_pX'$ with $X'$ at
$p$ is $2$: if $T_pX'$ meets $X'$ at another point, then all points of $T_pX'$ have $X'$-rank at most $2$. Thus we may assume
$\deg (T_pX\cap X)=2$ and supported at $p$. Since $X'$ is smooth at $p$ and of degree $d-1$, the linear projection $\pi_{T_pX'}: \PP^3\setminus T_pX'\to \PP^1$ away from the line $T_p X'$ extends to a degree $d-3$ morphism $\psi: X'\to \PP^1$. Since $d\ge 5$, a general fiber of $\psi$ has cardinality at least two. This is equivalent to say that a general element of $T_pX$ has $X'$-rank at most $2$. (An analogous algorithm can be found in \cite{BB}.)
\end{proof}

By Lemma \ref{o1}, we have $\textnormal{rk}_{X'}(b) \le 2$ for a general $b\in \Gamma_a$. Thus $\textnormal{rk}_{X'}(\pi _a(q)) \le 2$ for a general $q\in \Gamma$. 

Fix a general $q\in \Gamma\subset W_4$ and recall that $a'\in \PP^3$ is the image of $a$ upon taking the closure of the image $\pi_a(X\setminus \{a\})$. Take $S'\subset X'$ such that $\sharp S' \le 2$ and $\pi _a(q)\in \langle S'\rangle$. If $S'\subset
X'\setminus \{a'\}$, then we may take $S''\subset X\setminus \{a\}$ with
$\pi_a(S'')=S'$ and hence $q\in \langle S''\cup \{a\}\rangle$. Since $\sharp (S''\cup \{a\})\le 3$, we
derive $\textnormal{rk}_X(q)\le 3$, a contradiction. 

Now assume $a'\in S'$. We have $\sharp S' =2$, because otherwise $q\in T_aX$, contradicting the generality of
$q$. Set
$\{b'\}:= S'\setminus \{a'\}$. If $\sharp(\langle S'\rangle \cap X')\ge 3$, we may take $a''\in X'\setminus
\{a'\}$ such that $\pi_a(q)\in \langle \{a'',b'\}\rangle$ and derive $\textnormal{rk}_X(q)\le 3$. Thus we may assume $\langle
S'\rangle \cap X' =\{a',b'\}$ (set-theoretically).  

Let $\Ss (X',\pi _a(q))$ denote the set of rank-decompositions of $\pi_a(q)$ with respect to $X'$. Since any two different
lines through $a'$ meet only at $a'$, $S'$ is the only element of $\Ss (X,\pi _a(q))$ containing $a'$, because otherwise another one should have
contained $\pi_a(q)\neq a'$ as well. Thus, to conclude, it is enough to prove that the set $\Ss (X',\pi _a(q))$ contains
at least another decomposition. As mentioned above, since $\Gamma$ is not a cone with vertex $a$ (for a general $a\in X$), $\Gamma_a$ is a surface. So it is sufficient to prove that $|\Ss (X',\pi_a(q))| >1$ is satisfied by all points $\pi _a(q)\notin \Sigma$ (i.e., possibly outiside some curve $\Sigma$) having at least a decomposition in $\Ss (X',\pi_a(q))$ containing $a'$. It is sufficient to apply the next lemma.

\begin{lemma}\label{x1}
Assume $d\ge 5$. Let $C_{a'}(X') \subset \PP^3$ be the cone with vertex $a'$ and $X'$ as its base. There exists a curve $\Sigma \subset C_{a'}(X')$ such that for all points $x\in C_{a'}(X')\setminus \Sigma $, with $\textnormal{rk}_{X'}(x) =2$, either $\Ss (X',x)$ is infinite or $|\Ss (X',x))|>1$.
\end{lemma}

\begin{proof}
The surface $C_{a'}(X')$ is irreducible and $C_{a'}(X')\ne \tau (X')$. Therefore they intersect along a curve, $\dim \left(C_{a'}(X')\cap \tau (X')\right) =1$. 

Consider all $x\in C_{a'}(X')\setminus \left(\tau (X')\cap C_{a'}(X')\right)$, with $\textnormal{rk}_{X'}(x) =2$, and suppose $\Ss (X',x)$ is not infinite. 

Let $\pi_x: \PP^3\setminus \{x\}\to \PP^2$ denote the linear projection away from $x$. Since $x\notin \tau
(X')$, then $x\notin X'$ and $\pi _{x|X'}$ is a local embedding. We are assuming $\Ss (X',x)$ is finite, i.e., $\pi_{x|X'}$ is birational onto its image. Note that if there are at least three different points of $X'$ with the same image by $\pi _x$, then $|(\Ss (X',x))| \ge 3$. So this case is excluded. Likewise, if there are $t$ singular points
of $\pi _x(X')$, we have $|(\Ss (X',x))| \ge t$, because $\pi _{x|X'}$ is a local embedding and so distinguishes tangent directions. 

Thus we may assume that $\pi _x(X')$ has a unique singular point, $\alpha$, which has exactly two branches, each of them smooth (the case with more branches was excluded above). Since $\deg (\pi_x(X'))=d-1$, one has $p_a(\pi _x(X')) = (d-2)(d-3)/2$. Thus $\alpha$ is a tacnode with arithmetic genus $(d-2)(d-3)/2$, since the normalization of $X$ is rational. 
Hence there exist distinct $a_1,a_2\in X'$, such that the plane $\langle T_{a_1} X' \cup T_{a_2} X'\rangle $ contains $x$ and the order of contact of $T_{a_i} X'$ with $X'$ at $a_i$ is at least three.

However, since $X'$ is contained in a smooth quadric surface $G$, each tangent line of $X'$ with order of contact at least three is contained in $G$. One of the rulings of $G$ is formed by lines with degree of intersection $d-2 \ge 3$ with $X'$. Therefore we may take as $\Sigma$ a finite union of lines of $G$ and the curve $C_{a'}(X')\cap \tau (X')$.
\end{proof}

\begin{small}

\end{small}

\end{document}